\documentclass[11pt,a4paper,leqno]{amsart}
\usepackage[utf8]{inputenc}
\nonstopmode \textwidth=16.00cm
\usepackage{amssymb,color,tikz, cancel}
\usepackage{amsmath,amscd,amsfonts,amsthm,verbatim}
\usepackage{tikz-cd}
\usepackage{enumitem}
\usepackage[all]{xy}
\usetikzlibrary{matrix} 

\usepackage{xcolor}
\usepackage{blindtext}
\usepackage{multicol}

\numberwithin{equation}{section} 

\definecolor{ffzzqq}{rgb}{1,0.6,0}
\definecolor{qqqqff}{rgb}{0,0,1}
\definecolor{ffqqqq}{rgb}{1,0,0}
\definecolor{wwzzqq}{rgb}{0.4,0.6,0}
\definecolor{zzwwff}{rgb}{0.6,0.4,1}
\usepackage{pgfplots}
\pgfplotsset{compat=1.15}
\usepackage{mathrsfs}
\usetikzlibrary{arrows}

\newcommand{\CC}{\mathbb{C}}
\newcommand{\mA}{\mathbb{A}}
\newcommand{\NN}{\mathbb{N}}

\newcommand{\ZZ}{\mathbb{Z}} 
\newcommand {\PP}{\mathbb{P}}

\newcommand {\sO}{\mathcal{O}}

\newcommand{\depth}{\text {depth}}

\DeclareMathOperator{\rank}{rank}

\DeclareMathOperator{\Proj}{Proj}
 \def\cocoa{{\hbox{\rm C\kern-.13em
 o\kern-.07em C\kern-.13em o\kern-.15em A}}}
\newtheorem{theorem}{Theorem}[section]
\newtheorem{lemma}[theorem]{Lemma}
\newtheorem{proposition}[theorem]{Proposition}
\newtheorem{corollary}[theorem]{Corollary}

 \theoremstyle{definition}
\newtheorem{definition}[theorem]{Definition} \theoremstyle{remark}
\newtheorem{remark}[theorem]{Remark}
\newtheorem{example}[theorem]{Example}

\definecolor{MyDarkGreen}{cmyk}{0.7,0,1,0}

\begin{document}

\title [Quasi-homogeneous singularities of hypersurfaces and syzygies]
{Quasi-homogeneous singularities of projective hypersurfaces and Jacobian syzygies}
 \author[A. V. Andrade]{Aline V. Andrade}
 \address{ICEx - UFMG, Department of Mathematics, Av. Ant\^onio Carlos, 6627, 30123-970 Belo Horizonte, MG, Brazil}
 \email{andradealine@mat.ufmg.br, ORCID 0000-0001-7129-3953}
 \author[V. Beorchia]{Valentina Beorchia} 
 \address{Dipartimento di Matematica, Informatica e Geoscienze, Universit\`a di
Trieste, Via Valerio 12/1, 34127 Trieste, Italy}
 \email{beorchia@units.it, 
 ORCID 0000-0003-3681-9045}.
 \author[A. Dimca]{Alexandru Dimca}
\address{Universit\'e C\^ ote d'Azur, CNRS, LJAD, France and Simion Stoilow Institute of Mathematics,
P.O. Box 1-764, RO-014700 Bucharest, Romania}
\email{dimca@unice.fr, ORCID 0000-0001-9679-9870}

 \author[R.\ M.\ Mir\'o-Roig]{Rosa M.\ Mir\'o-Roig} 
 \address{Facultat de
 Matem\`atiques i Inform\`atica, Universitat de Barcelona, Gran Via des les
 Corts Catalanes 585, 08007 Barcelona, Spain} \email{miro@ub.edu, ORCID 0000-0003-1375-6547}

\thanks{The 1st and 2nd authors are supported by ICTP-INdAM Collaborative Grants and Research in Pairs Programme.
The 1st author is supported by CAPES/Print grant 88887.913035/2023-00 and CNPq universal grant 408974/2023-0. The 2nd author is a member of GNSAGA of INdAM, is supported by MUR funds: PRIN project 2022BTA242 {\it Geometry of algebraic structures: moduli, invariants, deformations}, and by the University of Trieste project FRA 2025. The 3rd author was supported by the project {\it Singularities and Applications} - CF 132/31.07.2023 funded by the EU - NextGenerationEU - through Romania's National Recovery and Resilience Plan. The fourth author was supported by the grant PID2020-113674GB-I00}


\medskip \noindent

\begin{abstract} We prove an unexpected general relation between the Jacobian syzygies of a projective hypersurface $V\subset \PP^n$ with only isolated singularities and the nature of its singularities.
This allows to establish a new method for the identification of quasi-homogeneous hypersurface isolated singularities.
The result gives an insight on how the geometry is reflected in the Jacobian syzygies and extends previous results of the first, second and last author for free and nearly free plane curves \cite{ABMR}.
\end{abstract}

\maketitle

\section{Introduction}
In this paper, we address the study of a general relation between the Jacobian syzygies of a projective hypersurface $V\subset \PP^n$ with only isolated singularities and the nature of its singularities.
In particular, we are interested in finding geometric conditions 
for an isolated hypersurface singular point to be quasi-homogeneous, that is when it can
be locally represented by a weighted homogeneous polynomial, see Definition \ref{def: quasi-homogeneous}. The definition turns out to be difficult to test in concrete cases, especially for singular points of high multiplicity.

Here we establish a useful and efficient criterion which involves the global Jacobian singular scheme, that is the scheme defined by the Jacobian ideal $J_f$, generated by the partial derivatives
 $f$. Our criterion relies on the Jacobian syzygies and can be expressed in terms of a first syzygy matrix appearing in a minimal free resolution of $J_f$. The main result is the following:

\begin{theorem}
Let $p$ be a singular point of the hypersurface $V=V(f)$ and
assume that $V$ has only isolated singularities. Then $(V,p)$ is a quasi-homogeneous singularity if and only if there is a Jacobian syzygy $\rho=(A_0, \cdots, A_n)$ for $f$ and an integer $k \in \{0,\cdots,n\}$ such that $A_k(p) \ne 0$.

Equivalently, if $M_f$ is a first syzygy matrix in a minimal free resolution of the Jacobian ideal $J_f$ of $f$,
then a singular point $p \in {\rm Sing}\ C$ is quasi-homogeneous if and only if ${\rm rk}\ M_f(p) \ge 1$.
 \end{theorem}

The proof relies on the construction of a suitable homogeneous Jacobian syzygy starting from K. Saito's characterization of a quasi-homogeneous isolated singularity given in \cite{S}.
Besides the geometric relevance of our result, we note that our criterion is very simple to test. Indeed, a first syzygy matrix is easily found by many symbolic computational programs, like Macaulay 2 \cite{M2} or Singular \cite{Sing}. Moreover, we can check if all the singular points of a given curve are quasi-homogeneous just by checking if the zero set of the ideal generated by the entries of $M_f$ is empty, even without knowing the coordinates of the singular points. We include several examples on which we test our criterion. We point out, however, that from a computational point of view, another simple test on the presence of non-quasi-homogeneous singularities can be performed by using 
Proposition \ref{rk2}, which relies on a result of Brian\c con and Skoda, see \cite{BS}.

In the planar case, by following the lines of \cite{ABMR}, we give an alternative and geometric proof of the theorem. We consider the polar map
$\nabla f: \PP^2 \dasharrow \PP^2$, defined by the three partials of $f$, and the surface $S_f$ given by the closure of its graph in $\PP^2 \times \PP^2$, that is the blow-up of the Jacobian scheme. Such a surface is contained in a surface
$Z_f \subset \PP^2 \times \PP^2$, which is obtained from a first syzygy matrix arising in a minimal free resolution of $J_f$, see \eqref{eq: equations of Z_f}. We prove that $Z_f$ is arithmetically Cohen-Macaulay by an explicit construction of a minimal free resolution, see Proposition \ref{lemma: sections of E}. This allows us to compute the numerical class of $Z_f$, which only depends on the degree of $f$ and on the total Tjurina number, see Lemma \ref{lemma: class of Z_f}.

Our result follows from the comparison between the numerical class of $Z_f$ and the one of $S_f$. It turns out that $S_f=Z_f$ if and only if
all singularities are quasi-homogeneous.

Finally, we point out that some of our arguments can be rephrased in terms of the Rees and the Symmetric algebra of the Jacobian ideal $J_f$, which are the coordinate rings of $S_f$ and $Z_f$, respectively. The Cohen-Macaulayness of the Symmetric algebra 
of homogeneous ideals and the question whether it coincides with the Rees algebra have been investigated, in general, in several papers (see \cite{EH, Lin, SUV, TPDA}).
In particular, for the Jacobian ideal, a different proof of the Cohen-Macaulayness of $Z_f$ and a minimal free resolution are given in \cite[Theorem 3.2.6]{B-C}. Moreover, the equivalence between the irreducibility of $Z_f$ and the property of having only quasi-homogeneous singular points can also be proved by using the results of \cite{NN} and \cite{NS}.

Moreover, there is a parallelism between the notions of degree and Hilbert-Samuel multiplicity of a non-reduced zero dimensional connected component of the base scheme of a projective linear system, and the Tjurina and Milnor numbers of an isolated singularity, and being locally a complete intersection translates into a quasi-homogeneous singularity.
The study of the difference between the two numbers is considered in \cite{BCJ} by comparing the Rees algebra and the symmetric algebra of the ideal defining a generically finite rational map from $\PP^n \dasharrow \PP^{n+1}$. By \cite[Theorem 1]{H}, the polar map of a hypersurface $V(f)$ with only isolated singularities is generically finite, unless $V(f)$ is a cone with vertex a point. When the two degrees are different, the authors prove in \cite[Proposition 5]{BCJ} that the additional component of the symmetric algebra is a hyperplane of multiplicity equal to the difference of the two degrees. We point out, however, that Theorem \ref{thm1}
holds in all cases.

Now we outline the structure of the paper. In Section 2 we fix the notation, the definitions of Milnor and Tjurina numbers, of quasi-homogeneous singularities, and we exhibit some examples to illustrate these concepts. 
In Section 3 we prove the main Theorem and illustrate its usefulness on several examples.
In Section 4 we focus on the case of reduced planar curves: we determine the classes of $S_f$ and $Z_f$, we identify a minimal free resolution of $Z_f$ and we prove our main theorem (see Theorem \ref{main2}). In the last section, we conclude with some remarks concerning curves admitting some non-quasi-homogeneous singular point; the residual scheme to $S_f$ can be non-reduced, and it would be interesting to bound its degree in terms of the degrees of the Jacobian syzygies. Furthermore, by considering the so-called {\it Koszul hull} of $S_f$ introduced in \cite{B-C} (see \eqref{eq: Koszul hull}),
one can also define a zero-dimensional scheme of degree $\mu(C)$ in $\PP^2$, which contains properly the Jacobian scheme, and whose properties could give a better insight into the nature of non-quasi-homogeneous singularities.

\medskip \noindent  {\bf Acknowledgement.} 
The first and second authors thank ICTP (Trieste) for the warm hospitality, where most of the research was performed.

\section{Notation and background}
 This section contains the basic definitions and results on Jacobian ideals associated with
reduced singular projective hypersurfaces and it lays the groundwork for the results in the later sections.

From now on, we fix the polynomial ring $R = \CC[x_0, x_1, \cdots ,x_n]$ and we denote by $V = V (f )$ a
hypersurface of degree $d$ in the complex projective space $\PP^n = \Proj(R)$ defined by a
homogeneous polynomial $f \in R_d$.

\subsection{Jacobian ideal of a hypersurface}

Given a hypersurface $V =V(f)$ of degree $d$, we define the {\em Jacobian ideal} $J_f$ of $V$ as the homogeneous ideal in $R$ generated by the $n+1$ partial derivatives
$\partial_0 f:=\frac{\partial f}{\partial x_0}$, $\cdots $, $\partial_n f:=\frac{\partial f}{\partial x_n}$, and the {\em Jacobian scheme} $\Sigma_f$ of $V=V(f) \subseteq \PP^n$ is the zero-dimensional scheme with homogeneous ideal $J_f$.

We denote by ${\rm Syz}(J_f)$ the graded $R$-module of all Jacobian relations for $f$, that is
\begin{equation}\label{eq: syzygy module}
{\rm Syz}(J_f):=\{(A_0,\cdots ,A_n)\in R ^{n+1}\mid A_0 \partial_0 f +A_1 \partial_1 f+ \cdots +A_n \partial_n f = 0 \}.
\end{equation} 

To any syzygy $\rho$, one can associate the following derivation
 \begin{equation}
\label{eq11}
D_{\rho}=A_0\frac{\partial}{\partial x_0}+\cdots + A_n\frac{\partial}{\partial x_n} \in Der(R),
\end{equation}
on the polynomial ring $R$, which satisfies $D_{\rho}(f)=0$, that is the derivation $D_{\rho}$ kills the polynomial $f$.

The notation ${\rm Syz}(J_f)_t$ will indicate the homogeneous part of degree $t$ of the graded $R$-module ${\rm Syz}(J_f)$; if $V(f)$ has only isolated singularities,
we have that for any $t \ge 0$, the $\CC$-vector space ${\rm Syz}(J_f)_t$ has finite dimension.
The minimal degree of a Jacobian syzygy for $f$ is the integer $d_1$ defined to be the smallest integer $t$ such that there
is a nontrivial relation
$ a_0 \partial_0 f + \cdots +a_n \partial_n f = 0$
with coefficients $a_i \in R_t$. More precisely, we have:
\[
d_1=\min\{ t\in \NN \mid {\rm Syz}(J_f)_t\ne 0\}.
\]
It is well known that $d_1 = 0$, that is the $n+1$ partials
$\partial_0 f$, $\partial_1 f$, $\cdots $, $\partial_n f$ are linearly dependent, if and only if $V$ is
a cone with vertex of point $p\in \PP^n$ of multiplicity $d$. In this last case, the point $p$ is the only singular point of $V$ and it is easily seen to be a quasi-homogeneous singular point, since the Jacobian scheme is a complete intersection.
Therefore, we will always assume that $d_1>0$.

\begin{definition} Let $V =V(f)\subset \PP^n$ be a hypersurface of degree $d$ with isolated singularities.
We say that $V$ is a {\em $m$-syzygy hypersurface} if ${\rm Syz}(J_f)$ is minimally generated by $m$ homogeneous syzygies, say $\rho_1,\rho_2, \cdots, \rho_m$ of degrees $d_i= \deg \rho_i$ ordered so that $d_1\leq d_2\leq \cdots \leq d_m $. These degrees are called the {\em exponents} of the hypersurface $V$ and $\rho_1,\rho_2, \cdots, \rho_m$ a {\em minimal set of generators} for ${\rm Syz}(J_f)$.

When $n=m=2$, the $R$-module ${\rm Syz}(J_f)$ is free of rank $2$, and the curve $C=V(f)\subset \PP^2$ is called {\em free}, see \cite{D, DS2, DS4}. 

In general, when $n=m$, the hypersurface $V(f) \subset \PP^n$ is called {\em free}, and its singular locus has pure codimension $1$, see \cite {DS7}; in particular, if $n\ge 3$, isolated singularities never occur on free hypersurfaces.

\vskip 4mm
Let us conclude this section by recalling some basic notions from singularity theory.
We fix $V=V(f)\subset \mA^n$ a reduced, not necessarily irreducible, hypersurface and we fix an isolated singular point $p\in V$. 
Let $\CC \{y_1, \cdots , y_n \}$ denote the ring of convergent power series.

\begin{definition}\label{def: Milnor and Tjurina}
Let $V=V(f)\subset \mA^n$ be a reduced hypersurface and $p=(0,\cdots ,0)\in V$ an isolated singularity. We define the {\em Milnor number} of $V$ in $p$ as
\begin{equation}\label{eq: local Milnor}
\mu _{(0,\cdots , 0)}(V) = \dim \CC \{y_1,\cdots ,y_n \}/ \langle \partial_{y_1} f,\cdots ,\partial_{y_n} f \rangle.
\end{equation}

 We define the {\em Tjurina number} of $V$ in $p$ as
\begin{equation}\label{eq: local Tjurina}
\tau _{(0,\cdots , 0)}(V) = \dim \CC \{y_1,\cdots ,y_n \}/ \langle \partial_{y_1} f,\cdots ,\partial_{y_n} f ,f\rangle.
\end{equation}

We clearly have $\tau_{(0,\cdots ,0)}(V)\le \mu _{(0,\cdots ,0)}(V)$.
To define $\mu_p(V)$ and $\tau_p(V)$ for an arbitrary point $p$, translate $p$ to the origin.
\end{definition}

\begin{definition}\label{def: quasi-homogeneous}
A singularity is {\it quasi-homogeneous} if and only if there exists a holomorphic change of variables so that the defining equation becomes weighted homogeneous. Recall that $f(y_1,\cdots ,y_n) =\sum c_{i_1\cdots i_n}y_1^{i_1}\cdots y_n^{i_n}$ is said to be {\em weighted homogeneous} if there exist rational numbers $\alpha_1,\cdots ,\alpha_n$ such that $\sum c_{i_1\cdots i_n}y_1^{i_1\alpha_1}\cdots y_n^{i_n\alpha _n}$ is homogeneous.
\end{definition}

\vskip 2mm
By Saito's results on weighted homogeneous polynomials, see \cite{S}, it follows that a hypersurface $V$ with an isolated singular point $p$ has a quasi-homogeneous singularity in $p$ if and only if $\mu_p(V)=\tau_p(V)$.

Given an affine or projective hypersurface $V$, we define
$$
\mu (V):= \sum _{p\in Sing(V)} \mu _p(V)
$$
to be the {\em total Milnor number}.

For a hypersurface $V$ with only isolated singularities, it holds:
$$
\deg J_f = \tau (V):= \sum _{p\in Sing(C)} \tau _p(V)
$$
where $J_f$ is the Jacobian ideal. We call $\tau (V)$ the {\em total Tjurina number} of $V$.

Examples of quasi-homogeneous and non-quasi-homogeneous singularities are given by the following.

\begin{example}\label{ex_QH}
\begin{enumerate}
 \item \label{ex_4syz_QH} The nodal cubic $C\subset \PP^2$ of equation $C=V(f)=V(x_1^2x_2-x_0^2(x_0+x_2))$ 
 has a simple node in $p=(0:0:1)$. The Jacobian ideal is
 \[
 J_f=(-3x_0^2-2x_0x_2,2x_1x_2,-x_0^2+x_1^2).\]

We have $\mu(C)=\mu_p(C)=1$ and the singularity is quasi-homogeneous; we can see directly that the germ $(C,p)$ is given by $g(x_0,x_1)=x_0^2-x_1^2$, which is weighted homogeneous.

\item \label{ex_4syz_NQH} The rational quintic $C\subset \PP^2$ of equation $C=V(f)=V(x_0 x_1^2 x_2^2 +x_1^5 + x_2^5)$ is singular only in $p=(1:0:0)$ and $J_f=(x_1^2 x_2^2,5 x_1^4+2 x_0 x_1 x_2^2,2 x_0 x_1^2 x_2+5 x_2^4)$. 
We have $\tau (C)=\tau _p(C)=10$. By \cite[Proposition 4.10]{DSer} one has $\mu(C)=\mu_p(C)=11$ and the singularity is non-quasi-homogeneous.

\item Consider the hypersurface in $V=V(f)\subset \mathbb{P}^n$ with $f=x_0^{a}x_1^{b}+x_2^d+\cdots +x_n^{d}$, with $d=a+b$ and $a,b\geq 2$. $V$ has singularities in $p=(1:0:\cdots:0)$ and $q=(0:1:0:\cdots:0)$, and the germs $(V,p)$ and $(V,q)$ are given by $g_p=x_1^b+x_2^d+\cdots+x_n^d$ and $g_q= x_0^a+x_2^d+\cdots+x_n^d$, respectively. Since $g_p$ and $g_q$ are weighted homogeneous, both singularities are quasi-homogeneous.

\item Consider the hypersurface in $V=V(f)\subset \mathbb{P}^3$ with $f=x_0^2x_3^3+x_1^4x_3+x_2^5-x_0x_1x_2x_3^2$. $V$ has singularities in $p=(0:0:0:1)$ and $q=(1:0:0:0)$, and the germs $(V,p)$ and $(V,q)$ are given by $g_p=x_0^2+x_1^4+x_2^5-x_0x_1x_2$ and $g_q= x_3^3+x_1^4x_3+x_2^5-x_1x_2x_3^2$, respectively. Since $g_p$ and $g_q$ are not weighted homogeneous, both singularities are non-quasi-homogeneous. \end{enumerate}
\end{example}


\section{Main result and several reformulations}

The goal of this section is to establish an efficient criterion to determine whether an isolated hypersurface singularity is quasi-homogeneous. The criterion will be in terms of the Jacobian syzygies and easy to test with several symbolic computation programs like Macaulay2 \cite{M2}.
Indeed, we have the following result.
\begin{theorem}
\label{thm1}
Let $p$ be a singular point of a projective hypersurface $V=V(f)\subset \PP^n$ and
assume that $V$ has only isolated singularities. Then $(V,p)$ is a quasi-homogeneous singularity if and only if there is a Jacobian syzygy $\rho=(A_0, \cdots, A_n)$ for $f$ and an integer $k \in \{0,\dots,n\}$ such that $A_k(p) \ne 0$.
In other words, the associated derivation $D_{\rho}$ does not vanish in $p$.
\end{theorem}

\begin{corollary}
\label{cor1}
Let $p$ be a singular point of a hypersurface $V=V(f)\subset \PP^n$ and
assume that $V$ has only isolated singularities. Then $(V,p)$ is not a quasi-homogeneous singularity if and only if for any Jacobian syzygy $\rho=(A_0, \cdots, A_n)$ for $f$ and any integer $k \in \{0,\dots,n\}$ one has $A_k(p)= 0$.
In other words, all the associated derivations $D_{\rho}$ vanish in $p$.
\end{corollary}

\proof
First notice that the claim in Theorem \ref{thm1} does not depend on the choice of the linear coordinate system on $\PP^n$. See the end of this proof for a simple example of the changes involved in the claim when we change the coordinates.

Since $V$ has only isolated singularities, there is a system of coordinates $x=(x_0, \cdots,x_n)$ on $\PP^n$ such that
\begin{enumerate}
\item the ideal
$I=(\partial_1 f, \cdots, \partial_n f)$ is a complete intersection defining a $0$-dimensional
subscheme $\Sigma \subset \PP^n$ and

\item
$p=(1:0: \cdots :0).$
\end{enumerate}
To get such a system of coordinates, it is enough to take the hyperplane $H_0=V(x_0)$ to be transversal to the hypersurface $V$ and, in particular, to avoid the singularities of $V$. Then the remaining coordinates $x_1, \cdots,x_n$ are chosen such that the condition (2) above holds.
Note that $p$ belongs to the support $|\Sigma|$ of this subscheme, which is a finite set.

We suppose first that $(V,p)$ is a quasi-homogeneous singularity.
Let $\mathcal{O}$ be the structure sheaf on $\PP^n$ and for any $q \in \PP^n$ denote by $\mathcal{O}_q$ the corresponding local ring, given by the fiber of $\mathcal{O}$ at $q$. Let $I_q \subset \mathcal{O}_q$ denote the ideal obtained as the fiber of the sheaf ideal $\mathcal{I}$ associated to the ideal $I$.

Since $(V,p)$ is a quasi-homogeneous singularity, it follows that the germ of the regular function
$\partial_0 f/x_0^{d-1}$ in the point $p$ is in $I_p.$
To see this, let $y=(y_1,\cdots ,y_n)$ be the affine coordinates on the principal open set given by $x_0 \ne 0$ in $\PP^n$, where $y_j=x_j/x_0$ for all
$j \in \{1,\dots, n\}$. We set $g(y)=f(1,y)$
and $g_j(y)=\partial_j f(1,y)$ for $j \in \{1,\dots, n\}$. Note that $g_j(y)$ is the partial derivative of $g(y)$ with respect to $y_j$.
We identify $I_p$ with the localization $S_{\mathfrak{m}}$ of the polynomial ring $S=\CC[y_1,\cdots,y_n]$ with respect to the maximal ideal $\mathfrak{m}=(y_1,\cdots,y_n)$.
Then the Euler relation for $f$, namely
 \begin{equation}
\label{eq2}
d\, f= \sum_{j=0}^nx_j \partial_j f,
 \end{equation}
yields
 \begin{equation}
\label{eq3}
d\, g(y)=\partial_0 f(1,y)+ \sum_{j=1}^ny_j\, g_j(y).
 \end{equation}
Recall Saito's characterization of isolated quasi-homogeneous singularities given in \cite{S}: the germ $g$ is quasi-homogeneous if and only if $g \in J_g$, where $J_g$ is the Jacobian ideal of $g$ in the local ring $S_{\mathfrak{m}}$. Note that Saito Theorem works in the rings of (formal) power series associated with $S_{\mathfrak{m}}$, but the passage from such rings to the ring $S_{\mathfrak{m}}=\mathcal{O}_p$ is a standard application of completion theory, see
\cite[Chapter 10 and especially Exercise 8]{AM}.

Therefore we get that $g \in I_p$. It follows by \eqref{eq3} that the germ
\[
\partial_0f(1,y)=\partial_0f/x_0^{d-1}
\]
is also in $I_p$.
Now let $h\in S$ be a homogeneous polynomial of a sufficiently high degree $N$, vanishing with high order in all the points $q \in |\Sigma|$, $q \ne p$ and such that $h(p) \ne 0$.
It is then clear that the germ
\[
h\, \partial_0f/x_0^{N+d-1}
\]
is in $I_q$
for all $q \in |\Sigma|$. This implies that $h \, \partial_0f$ belongs to the saturation $I'$ of the ideal $I$ with respect to the maximal ideal $(x_0, \cdots, x_n)$. Since $I$ is a complete intersection, it follows that $I'=I$ and hence
\[
h \, \partial_0f \in I,
\]
see also \cite[Section 2]{DBull}). In other words, we have a relation
\[
h\, \partial_0f=\sum_{j=1}^na_j\, \partial_jf,
\]
which, in turn, gives rise to a Jacobian syzygy
\[
\rho=(h,-a_1, \cdots,-a_n)
\]
for $f$ with $h(p) \ne 0$.

Now we assume, conversely, that we have a Jacobian syzygy for $f$
 \begin{equation}
\label{eq4}
\sum_{j=0}^na_j \, \partial_jf=0
\end{equation}
such that $a_k(p) \ne 0$ for at least one index $k \in \{0,\dots,n\}$.

\medskip

\noindent{\bf Case 1.} Suppose that $k=0$. Then with the above notations, we can rewrite equation \eqref{eq4} for $x_0=1$ and $x_j=y_j$ in the form
\[
b_0(y)\,(d\, g(y)-\sum_{j=1}^n y_j\, g_j(y))+\sum_{j=1}^nb_j(y)\, g_j(y)=0,
\]
where $b_j(y)=a_j(1,y)$. Since by assumption $b_0(0) \ne 0$, this last equality implies that $g(y) \in I_p$, and hence using again Saito's characterization of isolated quasi-homogeneous singularities given in \cite{S}, we conclude that $(V,p)$ is quasi-homogeneous.

\medskip

\noindent{\bf Case 2.} Suppose $a_0(p)=0$ and that $a_k(p)\ne 0$ for some $k>0$. We show that by a suitable change of coordinates on $\PP^n$, preserving the point $p$, we may get back to the situation described in Case 1.
Consider the coordinate change
 \begin{equation}
\label{eq5}
x_0=z_0+z_k \text{ and } x_j=z_j \text{ for any } j\ge 1,
\end{equation}
and let
\[
F(z)=f(z_0+z_k,z_1, \cdots, z_n).
\]
Then one has
\[
F_s(z):=\frac{\partial F}{\partial z_s}(z)=\partial_s f(z_0+z_k,z_1, \cdots, z_n)
\]
for any $s \ne k$ and

\[
F_k(z):=\frac{\partial F}{\partial z_k}(z)=\partial_0 f(z_0+z_k,z_1, \cdots, z_n)+\partial_k f(z_0+z_k,z_1, \cdots, z_n).
\]
Set
\[
c_j(z):=a_j(z_0+z_k,z_1, \cdots, z_n),
\]
for all $j \in \{0,\dots,n\}$. If we substitute $x_0$ by $z_0+z_k$ and any
$x_j$ by $z_j$ for $j>0$ in \eqref{eq4}, we get
\[
\sum_{j=0}^nc_j(z)\partial_j f(z_0+z_k,z_1, \cdots, z_n)=0.
\]
This relation can be rewritten as
\[
(c_0(z)-c_k(z))F_0(z)+ \sum_{j=1}^nc_j(z)F_j(z)=0. 
\]
Since $c_0(p)-c_k(p)=a_0(p)-a_k(p)\ne 0$, the last equality is a syzygy for $F$ with the first entry non vanishing in $p$. Using Case 1 we infer that $(W,p)$ is a quasi-homogeneous singularity, where $W=V(F)$.
Since the singularities $(V,p)$ and $(W,p)$ are isomorphic, it follows that
$(V,p)$ is a quasi-homogeneous singularity as well.
\endproof

We can reformulate Theorem \ref{thm1} in the terms of a minimal set of generators of the graded $R$-module of Jacobian syzygies of $f$ defined in (\ref{eq: syzygy module}), by observing that $Syz(J_f)=D_0(f)$, the $R$-module of all derivations killing $f$, as defined in \eqref{eq11}.

\begin{theorem}\label{syz}
Let $p$ be a singular point of the hypersurface $V=V(f) \subset \PP^n$ and
assume that $V$ has only isolated singularities. 

Then $(V,p)$ is a quasi-homogeneous singularity if and only if there is a generator  $\rho_j=(A^j_{0}, \cdots, A^j_{n})$ of $Syz(J_f)$ and an integer $k \in \{0,\dots,n\}$ such that $A^j_{k}(p) \ne 0$.
\end{theorem}

This result can moreover be reformulated in terms of the first syzygy matrix of a minimal free resolution of the Jacobian ideal $J_f$ of $f$ as follows. 

Let $M_f$ be the
$(n+1) \times m$-matrix with entries in $R$, such that the $j$-th column
is given by $(A^j_{0}, \cdots,A^j_{n})$, that is by the components of the $j$-th generating syzygy in a minimal set of generators for $Syz(J_f)$, for all $j\in \{1,\dots,m\}$. Then Theorem \ref{thm1} can be rephrased as follows, and extends the case of reduced free and nearly free plane curves, proved in \cite{ABMR}.

\begin{corollary} 
\label{corT1}
Let $p$ be a singular point of a hypersurface $V=V(f)$ in $\PP^n$ and
assume that $V$ has only isolated singularities. Then $(V,p)$ is a quasi-homogeneous singularity if and only if
\[
\rank M_f(p) \geq 1,
\]
where $M_f(p)$ is the $(n+1) \times m$-matrix with entries in $\CC$ obtained by evaluating the matrix $M_f$ in the point $p$.
\end{corollary}

To sum up, given a hypersurface $V=V(f)\subset \PP^n$ with only isolated singularities, we can consider the
ideal $I_f$ generated by the entries of a first syzygy matrix $M_f$ of the Jacobian ideal and the support $Q_f$
of the scheme in $\PP^n$ defined by $I_f$. If $p \in V$ is a singular point, and if we know the coordinates of $p$, then we can choose local coordinates at $p$ and apply K. Saito's criterium in \cite{S} to decide whether $(V,p)$ is a quasi-homogeneous singularity or not, or check whether $p \notin Q_f$.

If we don't know the coordinates of all singular points of $V$, we check whether the ideal $J_f+I_f$ has no zero in $\PP^n$, that is $J_f+I_f$ has finite codimension in the polynomial ring $R$.

Another possible approach for the computation of the total Milnor number is given by the following result.

\begin{proposition}
\label{rk2}

Let  $V=V(f)\subset \PP^n$ with only isolated singularities. By a possible change of coordinates, assume that the hyperplane $H_0=V(x_0)$ is such that $V_0=V \cap H_0$ is smooth. Set $g(y_1,\dots,y_n):=f(1,y_1,\dots,y_n)$, and
consider the ideals in $\mathbb{C}[y_1,\ldots,y_n]$
\[
I_i:=(g,g_1,\ldots,g_n), \qquad 
I_n:=(g^n,g_1,\ldots,g_n).
\]
Then we have
\[
\tau(V)= \dim \frac{\mathbb{C}[y_1,\ldots,y_n]}{I_1}, \qquad \mu(V)= \dim \frac{\mathbb{C}[y_1,\cdots,y_n]}{I_{n}}.
\]

\end{proposition}
\begin{proof}
  By assumption the singularities of $V$ are in the affine open set $U_0\cong \mathbb{C}^n$ given by
$x_0 \ne 0$, and  $y_j=x_j/x_0$ are the corresponding coordinates on $U_0$.
Then the polynomial 
$g$ is a singularity germ in each singular point, so the claim on the total Tjurina number immediately follows.

To determine the total Milnor number, observe that if $p$ is one of the singularities of $V$, we can apply
a result of Brian\c con and Skoda, see \cite[Corollary]{BS}, and obtain that
$g^{n}$ belongs to the ideal of the local ring $\mathcal{O}_p$ spanned by
$g_1,\cdots,g_n$. It follows that 
$$
\mu(V)= \dim \frac{\mathbb{C}[y_1,\cdots,y_n]}{I_{n}},
$$
where the ideal $I_{n}$ is spanned by $g^{n},g_1,\cdots,g_n$.
Observe that $g^{n}$ is needed, since we consider only the zeros of the ideal $J_g=(g_1,\ldots,g_n)$ in $U_0$ which are situated on $g=0$, that is on $V\cap U_0$.
\end{proof}

\begin{remark}
  Proposition \ref{rk2} gives us a simple way to decide whether all the singularities of $V$
are quasi-homogeneous, which is equivalent to having
$\mu(V)=\tau(V)$.

\end{remark}
Toh\u aneanu has shown in \cite[Theorem 2.2]{To} that a curve $C \subset \PP^2$ having only isolated singularities is free if and only it admits a Jacobian syzygy which is non-zero at any point of $\PP^2$.
The following result tells what happens when the curve is not free, and in fact this property holds in any dimension.

\begin{proposition}
\label{propP}
Let $V:f=0$ be a hypersurface in $\PP^n$ having only isolated quasi-homogeneous singularities. Then for any countable subset $Y$ in $\PP^n$,
there is a syzygy of $V$ of degree at most $d_m$,  the maximal exponent of $V$, which is non-zero at any point of the set $Y$. In particular $Y$ can be the singular set of $V$.
\end{proposition}

\proof
Consider the finite dimensional $\CC$-vector space
$$W={\rm Syz}(J_f)_{d_m}.$$
For any point $y \in Y$ we denote by $E_y$ the vector subspace of $W$ consisting of all the Jacobian syzygies of $V$ in $W$ which vanish at the point $y$.
If $y\in Y$ is a singular point of $V$, then Theorem \ref{thm1} in the form given in Corollary \ref{corT1} shows that there is a syzygy in $W$ not vanishing at $y$ and hence we have $E_y \ne W$.

If $y\in Y$ is not a singular point of $V$, at least one of the partial derivatives $\partial_jf$ does not vanish at $y$.
It follows that one of the syzygies $\rho_{ij}$ of degree $(d-1)$, associated to the Koszul relation
$$\partial_jf  \partial_if - \partial_if  \partial_jf =0$$
is non zero at $y$. Note that we have $(d-1) \leq d_m$, since $d_m$ is the maximal exponent of $V$.
Hence in this case we also have $E_y \ne W$.

Now applying Baire Theorem or some basic facts of Measure Theory,
we see that
$$E=\cup _{y \in Y}E_y \ne W.$$
It follows that any syzygy in $W \setminus E$ satisfies our claim.
\endproof

We end this section with a final example concerning hypersurfaces in $\PP^n$.

\begin{example}\label{Chebyshev_Hyp}{\bf Chebyshev hypersurfaces.} 
Consider the hypersurface $C(n, d, k)\subset \PP^n$ for $n\geq 2$ defined by the homogeneous equation $f(n, d, k) = 0$, where the polynomial $f(n, d, k)$ is the homogenization (using $x_0$) of the polynomial
 \begin{equation}
 g(n,k,d)= T_d(x_1)+\cdots +T_d(x_n)+ k,
 \end{equation}
and $T_d$ is the $d$-th Chebyshev polynomial. 
It is easy to see that the hypersurface $C(n, d, k)$ is smooth, unless $k$ is an integer satisfying $|k|\leq n$ and $n+k$ is even. By \cite{DS1} if these two conditions are fulfilled, then the hypersurface $C(n, d, k)$ is nodal, and the number of nodes is given by 
\begin{equation}
 \mathcal{N}(C(n, d, k))= \left\{ \begin{array}{ll}
 {n\choose a}d_1^n,& d=2d_1+1 \\
 {n\choose a}d_1^n\left(1-\frac{1}{d_1}\right)^a,& d=2d_1+1
 \end{array} \right.
\end{equation}
when $n+k=2a$.

For $d$ odd the maximal number of nodes is obtained for $a =\left[\frac{n}{2}\right]$. Although for $d$ even, it is not clear for which $k$ the maximum of $\mathcal{N}(C(n, d, k))$ is attained, one may show that for $d \geq n+2$, the maximum is again attained for $a =\left[\frac{n}{2}\right]$.

We will call the {\em Chebyshev hypersurface} $C(n, d)$ the hypersurface corresponding to $a =\left [\frac{n}{2}\right]$, $k = 0$ for $n$ even, and $k=1$ for $n$ odd. Consider, for instance, the nodal Chebyshev hypersurface $C(4,6)=V(f)$ with
\begin{equation}
 f=32(x_1^6+x_2^6+x_3^6+x_4^6)-48x_0^2(x_1^4+x_2^4+x_3^4+x_4^4)+18x_0^4(x_1^2+x_2^2+x_3^2+x_4^2)-4x_0^6
\end{equation}

\noindent which is a $20$-syzygy hypersurface in $\PP^4$ with $\mathcal{N}(4,6,0)=216$ nodes. As expected $\rank M(f)(p)\geq 1$ for all the $216$ nodes.

\end{example}

\begin{example}
    Consider the $8$-syzygy hypersurface $V=V(f)\subset \mathbb{P}^3$ with $f=x_0^3x_3+x_1^4+x_1x_2^2x_3+x_0x_1^3$. The only singularities are in $p=(0:0:0:1)$ and $q=(0:0:1:0)$; moreover we have $\rank M_f(p)=0$ and $\rank M_f(q)=1$, therefore $p$ is a non-quasi-homogeneous singularity, while $q$ is quasi-homogeneous. 
\end{example}
\begin{example}
    Consider $V=V(f)\subset \mathbb{P}^n$ with $f=x_0^2x_n^3+x_1^4x_n+(x_2^2+x_3^2+\cdots +x_{n-1}^2)-x_0x_1x_n^2(x_2+x_3+\cdots +x_{n-1})$; the singularities are in $p=(1:0:\cdots:0)$ and $q=(0:\cdots:0:1)$, with $\rank M_f(p)=\rank M_f(q)=0$, therefore both singularities are non-quasi-homogeneous. 
\end{example}


\section{A geometric characterization of quasi-homogeneous singularities of reduced plane curves}

In this section, we focus on reduced plane curves $C = V (f )\subset \PP^2$ and using polar
maps we will give an alternative and more geometric proof of the above characterization of quasi-homogeneous singularities
on curves (see Theorem \ref{main2}) extending to {\em any} reduced plane curve our previous techniques and results on free and nearly free plane curves given in \cite[Theorem 3.3 and Theorem 5.6]{ABMR}.

 In this framework, the symmetric algebra $S(J_f)=\oplus _{n\ge 0}Sym^nJ_f$ of the Jacobian ideal will play an important role. In particular, we will strongly use the Cohen-Macaulayness of $S(J_f)$ which follows from \cite[Corollary 10.4]{HSV}. In fact, we have:

 \begin{theorem} \label{CM}
 Let $I\subset R=\CC[x_0,\cdots ,x_n]$ be a homogeneous ideal of height $>0$. Assume:
 \begin{itemize}
 \item[(a)] $I$ is an almost complete intersection,
 \item[(b)]
 $\depth (R/I)\ge \dim(R/I)-1.$
 \end{itemize}
 \end{theorem}
 Then $Sym(I)$ is a graded Cohen-Macaulay ring.
\begin{proof} It follows from \cite[Corollary 10.4]{HSV} taking into account that a graded ring is Cohen-Macaulay if and only if its localization at the maximal homogeneous ideal is Cohen-Macaulay (see \cite[Proposition 4.10]{HR}). 
\end{proof}

\vskip 4mm

Recall that for $n=2$ and $m\geq 3$ a minimal free resolution of $J_f$ looks like:
\begin{equation}\label{eq: m syzygy sequence}
0 \longrightarrow \bigoplus_{j=1}^{m-2} R(d-1-e_j) 
\overset {P_f}{\longrightarrow} \bigoplus_{i=1}^m R(-d_i) \overset {M_f}{\longrightarrow} R^3 \longrightarrow J_f (d-1)\longrightarrow 0,
\end{equation}

\medskip \noindent \medskip with $e_1\leq \cdots \leq e_{m-2}$ and $d_1\leq d_2\leq \cdots \leq d_m $.
\end{definition}

From the exact sequence (\ref{eq: m syzygy sequence}) we get:
\[
d-1-\sum_{i=1}^md_i=\sum_{j=1}^{m-2}(d-1-e_j),
\]
and
\begin{equation}\label{tau_C}\displaystyle \deg(J_f)=(d-1)^2-\sum_{1\leq i<j\leq m}d_id_j+\sum_{1\leq i<j\leq m-2}(d-1-e_i)(d-1-e_j)+(1-d)\sum_{j=1}^{m-2}(d-1-e_j).
\end{equation}

Moreover, it follows from \cite{D} and \cite{HS} that
\begin{equation}\label{def e_j}
 e_j=d+d_{j+2}-1+\epsilon _j,
\end{equation}
for $j=1,\cdots , m-2$ and some integers $\epsilon_j\geq 1$. This provides 
\begin{equation}\label{def epsilon_j}
 d_1+d_2=d-1+ \sum_{i=1}^{m-2} \epsilon _j.
\end{equation}

\begin{example}\label{ex_msyz}
\begin{enumerate}
\item \label{ex_3syz} The rational quintic $C\subset \PP^2$ of equation $C=V(f)=V(x_0^2 x_1^3+3 x_1^5-4 x_0 x_1^3 x_2+x_0^2 x_1 x_2^2)$ is a $3$-syzygy curve with two singular points in $p=(0:0:1)$ and $q=(1:0:0)$. The Jacobian ideal is $J_f=(2 x_0 x_1^3-4 x_1^3 x_2+2 x_0 x_1 x_2^2,3 x_0^2 x_1^2+15 x_1^4-12 x_0 x_1^2 x_2+x_0^2 x_2^2,-4 x_0 x_1^3+2 x_0^2 x_1 x_2 
)$ and it has a minimal free $R$-resolution of the following type:
\[
0 \longrightarrow R(-9) \overset {P_f}{\longrightarrow} R(-8)\oplus R(-7)\oplus R(-6) \overset {M_f}{\longrightarrow} R(-4)^3 {\longrightarrow} J_f \longrightarrow 0, \]
with 
\begin{small}
\[ 
M_f=\begin{pmatrix}
 6 x_0^2-5 x_0 x_2 & 12 x_0 x_1^2-5 x_0 x_2^2 & 30 x_1^4-27 x_0 x_1^2 x_2+5 x_0 x_2^3\\
-4 x_0 x_1 & -4 x_0 x_1 x_2 & 8 x_1^3 x_2\\ 
-15 x_1^2-4 x_0 x_2+5 x_2^2 & 6 x_0 x_1^2-27 x_1^2 x_2+2 x_0 x_2^2+5 x_2^3 & 15 x_1^4+18 x_1^2 x_2^2-5 x_2^4 \\
 \end{pmatrix},
\]
\end{small}

\medskip \noindent \medskip and 

\medskip
\begin{center}
$P_f=\begin{pmatrix}
 27 x_1^2 x_2-5 x_2^3\\ 
 -15 x_1^2+5 x_2^2\\ 
 6 x_0 \\
 \end{pmatrix}.
$
\end{center}

\medskip
\noindent Therefore, the exponents of the curve are $(2,3,4)$ and $\deg(J_f)=10$.

\medskip

\item \label{ex_4syz} The rational sextic $C\subset \PP^2$ of equation $C=V(f)=V((x_0^2+x_1^2)^3-4 x_0^2 x_1^2 x_2^2)$ is a $4$-syzygy curve which is singular in $p=(0:0:1)$, $q=(1:-i:0)$ and $r=(1:i:0)$. The Jacobian ideal is given by 
\[
J_f=(6 x_0^5+12 x_0^3 x_1^2+6 x_0 x_1^4-8 x_0 x_1^2 x_2^2,6 x_0^4 x_1+12 x_0^2 x_1^3+6 x_1^5-8 x_0^2 x_1 x_2^2,-8 x_0^2 x_1^2 x_2) )
\]
and it has a minimal free $R$-resolution of the following type:
\[
0 \longrightarrow R(-10)^2 \overset {P_f}{\longrightarrow} R(-9)^3\oplus R(-8) \overset {M_f}{\longrightarrow} R(-5)^3 {\longrightarrow} J_f \longrightarrow 0, 
\]
with 
\begin{small}
\[
M_f= \begin{pmatrix}
 -x_0 x_1^2 & -3 x_0^3 x_1+4 x_0 x_1 x_2^2 & -3 x_1^4& 4 x_0 x_1^2 x_2\\
x_0^2 x_1 & 3 x_0^4 & 3 x_0 x_1^3-4 x_0 x_1 x_2^2& 0\\
-x_0^2 x_2+x_1^2 x_2 & 9 x_0^2 x_1 x_2+3 x_1^3 x_2-4 x_1 x_2^3 & -3 x_0^3 x_2-9 x_0 x_1^2 x_2+4 x_0 x_2^3 & 3 x_0^4+6 x_0^2 x_1^2+3 x_1^4-4 x_1^2 x_2^2\\
 \end{pmatrix},
 \] \end{small}

\medskip \noindent \medskip and 

\begin{center}
$P_f=\begin{pmatrix}
 3 x_1^2-4 x_2^2 & 3 x_0^2\\
 0 & -x_1\\
 -x_0 & 0\\
 -x_2 & x_2\\
 \end{pmatrix}.$
\end{center}

\medskip \noindent \medskip The exponents of the curve are $(3,4,4,4)$ and $\deg(J_f)=16$.

\item \label{ex_5syz} The rational sextic $C\subset \PP^2$ of equation $C=V(f)=V(x_1^6+x_0^2 x_1^2 x_2^2+x_2^6)$ is a $5$-syzygy curve which is singular in $p=(1:0:0)$. We have 
\[
J_f=(2 x_0 x_1^2 x_2^2, 6 x_1^5+2 x_0^2 x_1 x_2^2, 2 x_0^2 x_1^2 x_2+6 x_2^5) 
)
\]
and it has a minimal free $R$-resolution of the following type:
\[
0 \longrightarrow R(-11)^3 \overset {P_f}{\longrightarrow} R(-10)^3\oplus R(-9)^2 \overset {M_f}{\longrightarrow} R(-5)^3 {\longrightarrow} J_f \longrightarrow 0,
\]
with 
\begin{small}
\[M_f= \begin{pmatrix}
 3 x_1^4+x_0^2 x_2^2 & x_0^2 x_1^2+3 x_2^4 & x_0^4 x_1-9 x_1^3 x_2^2 & 0 & x_0^4 x_2-9 x_1^2 x_2^3\\
-x_0 x_1 x_2^2 & 0 & 3 x_0 x_2^4 & -x_0^2 x_1^2 x_2-3 x_2^5 & -x_0^3 x_1 x_2\\
0 & -x_0 x_1^2 x_2 & -x_0^3 x_1 x_2 & 3 x_1^5+x_0^2 x_1 x_2^2 & 3 x_0 x_1^4\\
 \end{pmatrix},
 \] \end{small}

\medskip \noindent \medskip and 

\begin{center}
$P_f=\begin{pmatrix}
 -3 x_2^2 & x_0^2 & 0\\
 x_0^2 & -3 x_1^2 & 0\\
 -x_1 & 0 & x_2\\
 0 & 0 & x_0\\
 0 & -x_2 & -x_1\\
 \end{pmatrix}.
$
\end{center}

\medskip \noindent \medskip The exponents of the curve are $(4,4,5,5,5)$ and $\deg(J_f)=12$.
\end{enumerate}

\end{example}

Consider a $m$-syzygy reduced plane curve $C=V(f)\subset \PP^2$, fix $m$ generators of
the syzygy module 
${\rm Syz}(J_f)$ and denote them by
\[
\rho_1=(A^1_0, A^1_1, A^1_2), \quad \rho_2=(A^2_0,A^2_1,A^2_2), \quad \cdots , \quad \rho_m=(A^m_0,A^m_1,A^m_2),
\]
where $A^1_i \in R_{d_1}$, $A^2_i\in R_{d_2}, \cdots , A^m_i \in R_{d_m}$. Therefore,
a first syzygy matrix $M_f$ of a minimal free $R$-resolution (\ref{eq: m syzygy sequence}) of the Jacobian ideal $J_f$ looks like:
\begin{equation}\label{eq: 3 syzygy first syzygy matrix}
 M_f=\begin{pmatrix}
 A^1_0 & A^2_0 & \cdots & A^m_0\\
 A^1_1 & A^2_1 & \cdots & A^m_1 \\
 A^1_2 & A^2_2 & \cdots & A^m_2 \\
 \end{pmatrix},
\end{equation}
and the second syzygy matrix $P_f$ of a minimal free $R$-resolution of $J_f$ will be indicated as follows:
\begin{equation}\label{eq: 3 syzygy second syzygy matrix}
 P_f=\begin{pmatrix}
 P^1_1 & \cdots & P^{m-2}_1 \\
 P^1_2 & \cdots & P^{m-2}_2\\
 \vdots & \cdots & \vdots\\
 P^1_m & \cdots & P^{m-2}_m\\
 \end{pmatrix},
\end{equation}
where $P_i^j \in R_{1+e_j-d-d_i}$.

 Consider the {\it polar map}
$$
\nabla f : \PP^2 \dasharrow \PP^2, \quad p\to (\partial_0 f(p),\partial_1 f(p),\partial_2 f(p));
$$
we would like to describe the closure $S_f$ of the graph $\Gamma_{\nabla f}$.

We recall that (see, for instance, \cite[Lemma 3.1]{ABMR}),
the class of $S_f={\overline \Gamma}_{\nabla f}$ in the Chow ring $A (\PP^2 \times \PP^2)$ is given by
\begin{equation}\label{eq: class of Sf}
S_f\equiv (\deg \nabla f)\ h_1^2
+ (d-1)h_1 h_2 + h_2^2,
\end{equation}
where $\deg \nabla f= (d-1)^2 - \mu(C)$.

 Moreover, we observe that given a point
$p \in \PP^2 \setminus \Sigma_f$, where $\Sigma_f$ is the Jacobian scheme, and $q\in \PP^2$, if $q=(q_0:q_1:q_2)=\nabla f(p)$ then the following equations are satisfied:
\begin{equation}
\left\{
 \begin{array}{rl}
 A^1_0 (p) q_0 +A^1_1 (p)q_1 + A^1_2 (p)q_2&=0\\
 A^2_0 (p) q_0 +A^2_1 (p)q_1 + A^2_2 (p)q_2&=0\\
 \vdots & \\
 A^m_0 (p) q_0 +A^m_1 (p)q_1 + A^m_2 (p)q_2&=0.\\
 \end{array}
 \right .
\end{equation}
As a consequence the closure
$S_f = \overline {\Gamma}_{\nabla f}\subset \PP^2 \times \PP^2$ is contained in the locus
\begin{equation}\label{eq: equations of Z_f}
\left\{
 \begin{array}{cl}
 A^1_0 y_0 +A^1_1 y_1 + A^1_2 y_2&=0\\
 A^2_0 y_0 +A^2_1 y_1 + A^2_2 y_2&=0\\
 \vdots & \\
 A^m_0 y_0 +A^m_1 y_1 + A^m_2 y_2&=0\\
 \end{array}, 
 \right .
\end{equation}
that is in the intersection of $m$ divisors in $\PP^2\times \PP^2$: 
$$
D_1 \sim d_1 h_1 + h_2, \ D_2 \sim d_2 h_1 + h_2, \ \cdots , D_m \sim d_m h_1 + h_2. 
$$
These $m$ divisors determine a surface $Z_f$ in
$\PP^2 \times \PP^2$. To determine its class we first explicitly compute a minimal free resolution of
the ideal sheaf $\mathcal{I}_{Z_f, \PP^2 \times \PP^2}$.

We now prove that $Z_f$ is a Cohen-Macaulay scheme by finding explicitly a minimal free resolution.
\begin{remark}
 The Cohen-Macaulayness of the coordinate ring of $Z_f$ or, equivalently, of the Symmetric algebra of the Jacobian ideal $J_f$ as well as its minimal free resolution, have been determined in \cite[Theorem 3.2.6,]{B-C}. For the sake of completeness, we underline an alternative approach that is more direct and very explicit.
 \end{remark}

\begin{proposition}\label{lemma: sections of E} 
The surface $Z_f\subset \PP^2\times \PP^2$ is arithmetically Cohen-Macaulay and its ideal
 sheaf $\mathcal{I}_{Z_f, \PP^2 \times \PP^2}$ 
 admits a resolution of the following type:
\begin{equation}\label{eq: sequence defining E}
0\longrightarrow \bigoplus _{j=1}^{m-1}\sO_{\PP^2 \times \PP^2} (-B_j)\longrightarrow \bigoplus _{i=1}^m\sO_{\PP^2 \times \PP^2} (-D_i) \longrightarrow \mathcal{I}_{Z_f, \PP^2 \times \PP^2} \longrightarrow 0,
\end{equation}
where $D_i\sim d_ih_1+h_2$ for $1\le i \le m$, $B_j\sim (1+e_j-d) h_1+h_2$ for $1\le j\le m-2$ and $B_{m-1}\sim (d-1)h_1+2h_2$. 
\end{proposition}
\begin{proof} 
Since the coordinate ring of $Z_f$ is the symmetric algebra $S(J_f)=\oplus _{n\ge 0}Sym^nJ_f$ of the Jacobian ideal $J_f$ of the plane curve $C$, applying Theorem \ref{CM} we deduce that the surface $Z_f\subset \PP^2\times \PP^2$ is arithmetically Cohen-Macaulay. Therefore, according to the Hilbert-Burch theorem, its $m$ minimal generators $D_i\sim d_i \ h_1+ h_2$ are defined by the maximal minors of a $m \times (m-1)$ matrix $S$ which we will now describe.

To this end, let us compute a minimal free resolution of $\mathcal{I}_{Z_f, \PP^2 \times \PP^2}$ which is equivalent to explicitly determine $m-1$ linearly independent syzygies of the divisors $D_1,\cdots, D_m$.

 Observe that the exactness of the sequence 
\eqref{eq: m syzygy sequence} implies that
\begin{equation}\label{eq: first section of E}
 M_f \cdot P_f = 0.
\end{equation}

As a consequence, for any integer $1\le j \le m-2$, we also have:
\begin{equation}\label{eq: syzygy in p^2xp^2}
P_1^j (A_0^1 y_0 +A_1^1 y_1 + A_2^1 y_2)
+P_2^j(A_0^2 y_0+A_1^2 y_1+ A_2^2 y_2)+ \cdots +
P_m^j(A_0^m y_0 +A_1^m y_1 + A_2^m y_2) \equiv 0.
\end{equation}
This gives us the first $m-2$ syzygies of $\mathcal{I}_{Z_f, \PP^2 \times \PP^2}$.

We are now looking for the remaining syzygy. To this end, we first observe that we can write the defining equations \eqref{eq: equations of Z_f} in the form
 \begin{equation}
 \begin{pmatrix}
 y_0 & y_1 & y_2\\
 \end{pmatrix}
 \cdot M_f =
 \begin{pmatrix}
 0 & 0 & \cdots & 0\\
 \end{pmatrix}.
 \end{equation}

 Moreover, since the columns $\rho_i:=\begin{pmatrix}
 A^i_0 \\ A^i _1 \\ A^i_2
 \end{pmatrix}$ of $M_f$, $1\le i \le m$, are a minimal system of generators of
 ${\rm Syz}(J_f)$, we can express the degree $d-1$ Koszul syzygies of $J_f$
 \[
 (0, -\partial_2 f,\partial_1 f), \quad (\partial_2 f, 0, -\partial_0 f),
 \quad (-\partial_1 f, \partial_0 f,0)
 \]
 as linear combinations of the following type:
 
 \[
 \begin{matrix}
 (0, -\partial_2 f,\partial_1 f)&=&n_{10}\rho_1 +n_{20}\rho_2+\cdots +n_{m0}\rho_m, \\
 (\partial_2 f, 0, -\partial_0 f)&=&n_{11}\rho_1 +n_{21}\rho_2+\cdots +n_{m1}\rho_m,\\
 (-\partial_1 f, \partial_0 f,0)&=&n_{12}\rho_1 +n_{22}\rho_2+\cdots +n_{m2}\rho_m\\
 \end{matrix}
 \]
 for suitable homogeneous polynomials $n_{ij}=n_{ij}(x_0,x_1,x_2)$ with $1\le i \le m$ and $0\le j \le 2$.

 In particular, by setting $N = (n_{ij})$, where $1\le i\le m$, and $0\le j\le 2$, we have
 \begin{equation}\label{eq: M.N}
 M_f \cdot N =
 \left(
 \begin{array}{rrr}
 0& \partial_2 f & -\partial_1 f \\
 -\partial_2 f & 0 & \partial_0 f \\
 \partial_1 f & - \partial_0 f &0 \\
 \end{array}
 \right),
 \end{equation}
 which is a skew-symmetric matrix. As a consequence, we have for any 
 $(y_0:y_1:y_2) \in \PP^2$
 \begin{equation}\label{eq: second section}
 \begin{pmatrix}
 y_0 & y_1 & y_2\\
 \end{pmatrix}
 \cdot M_f \cdot N \cdot \begin{pmatrix}
 y_0 \\ y_1 \\ y_2\\
 \end{pmatrix}=0.
 \end{equation}
This implies that the entries of the column matrix 
\[
\mathcal{P}=N \cdot \begin{pmatrix}
 y_0 \\ y_1 \\ y_2\\
 \end{pmatrix}
\]
give a non-zero syzygy between the $m$ forms $\begin{pmatrix}
 y_0 & y_1 & y_2\\
 \end{pmatrix}
 \cdot M_f$, and the class of the relation \eqref{eq: second section},
 as a divisor in $\PP^2 \times \PP^2$, is $(d-1)h_1+2h_2$.

Putting it all together, we get a map
 \begin{equation}
 \bigoplus _{j=1}^{m-1}
 \sO_{\PP^2 \times \PP^2}(-B_j) \overset {S}{\longrightarrow}
 \bigoplus_{i=1}^m \sO_{\PP^2 \times \PP^2} (-D_i),
 \end{equation}
 where $D_i=d_ih_1+h_2$ with $1\le i \le m$, $B_j=(1+e_j-d) h_1+h_2$ for $1\le j\le m-2$ and $B_{m-1}=(d-1)h_1+2h_2$. Therefore the matrix $S$ can be written as 

 \begin{equation}\label{eq: definition of matrix S}
 S=
 \begin{pmatrix}
 P_f & \mathcal{P}\\
 \end{pmatrix}.
 \end{equation}
\end{proof} 
\begin{lemma}
 The matrix $S$ of \eqref{eq: definition of matrix S} is the Hilbert-Burch matrix of $Z_f$.
\end{lemma} 
\begin{proof} We have to verify that
the order $m-1$ minors of the matrix $S$ are proportional, by a nonzero constant, to the $m$ forms $\rho_1 \cdot \begin{pmatrix}
 y_0\\ y_1\\ y_2\\
 \end{pmatrix}$, $\rho_2 \cdot \begin{pmatrix}
 y_0\\ y_1\\ y_2\\
 \end{pmatrix}$, $\cdots $, 
 $\rho_m \cdot \begin{pmatrix}
 y_0\\ y_1\\ y_2\\
 \end{pmatrix}$.
 Indeed, by \eqref{eq: first section of E} we have $M_f \cdot P_f=\begin{pmatrix}
 0\\ 0\\ 0\\
 \end{pmatrix}$, and by \eqref{eq: M.N} it is $M_f \cdot N= K$, where $K$ is the matrix of Koszul relations:
 \[
 K=\left(
 \begin{array}{rrr}
 0& \partial_2 f & -\partial_1 f \\
 -\partial_2 f & 0 & \partial_0 f \\
 \partial_1 f & - \partial_0 f &0 \\
 \end{array}
 \right).
 \]
 This gives
 \begin{equation}\label{eq: relations P N M}
 \left(
 \begin{array}{c}
 {}^t P_f\\
 {}^t N\\
 \end{array}
 \right) \cdot {}^t M_f \cdot \begin{pmatrix}
 y_0\\ y_1\\ y_2\\
 \end{pmatrix}=
 \left(
 \begin{array}{rrr}
 0&0&0\\
 0& -\partial_2 f & \partial_1 f \\
 \partial_2 f & 0 & -\partial_0 f \\
 -\partial_1 f & \partial_0 f &0 \\
 \end{array}
 \right)\cdot \begin{pmatrix}
 y_0\\ y_1\\ y_2\\
 \end{pmatrix}. 
 \end{equation}
 Observe now that 
 \[
 S= \begin{pmatrix}
 1&0&0&0\\
 0&y_0 &y_1 &y_2\\
 \end{pmatrix}\cdot \left(
 \begin{array}{c}
 {}^t P_f\\
 {}^t N\\
 \end{array}
 \right),
 \]
 and furthermore, we have
 \[
 \begin{pmatrix}
 1&0&0&0\\
 0&y_0 &y_1 &y_2\\
 \end{pmatrix}\cdot \left(
 \begin{array}{rrr}
 0&0&0\\
 0& -\partial_2 f & \partial_1 f \\
 \partial_2 f & 0 & -\partial_0 f \\
 -\partial_1 f & \partial_0 f &0 \\
 \end{array}
 \right)\cdot \begin{pmatrix}
 y_0\\ y_1\\ y_2\\
 \end{pmatrix}
 = \begin{pmatrix}
 0\\0\\
 \end{pmatrix}.
 \]
 Therefore, by left multiplying $\begin{pmatrix}
 1&0&0&0\\
 0&y_0 &y_1 &y_2\\
 \end{pmatrix}$ 
 to both members of \eqref{eq: relations P N M}, we obtain
 \[
 S \cdot {}^t M_f \cdot \begin{pmatrix}
 y_0\\ y_1\\ y_2\\
 \end{pmatrix}=\begin{pmatrix}
 0\\0\\
 \end{pmatrix},
 \]
 so the order $m-1$ minors of $S$ are indeed proportional to the $m$ entries of ${}^t M_f \cdot \begin{pmatrix}
 y_0\\ y_1\\ y_2\\
 \end{pmatrix}$.
 \end{proof}

 We are now ready to determine the class of $Z_f$. In fact, we have:

 \begin{lemma} \label{lemma: class of Z_f}
 Keeping the above notation, we have
 \begin{equation}\label{eq: class of ci}
 Z_f \equiv ((d-1)^2 -\tau(C) )\ h_1^2 +(d-1) \ h_1 \ h_2 + h_2^2. 
 \end{equation}
 
\end{lemma}
\begin{proof} 
 Since $Z_f$ is an arithmetically Cohen-Macaulay codimension two subscheme of $\PP^2\times \PP^2$, it is of pure dimension $2$, so its class can be written in the form
 $$
 Z_f \equiv \alpha \ h_1^2 +\beta \ h_1 \ h_2 + \gamma \ h_2^2,
 $$ 
 for some coefficients $\alpha, \beta, \gamma \in \ZZ$. Using the exact sequence
$$ 
0\longrightarrow \bigoplus _{j=1}^{m-1}\sO_{\PP^2 \times \PP^2} (-B_j) \longrightarrow \bigoplus _{i=1}^m\sO_{\PP^2 \times \PP^2} (-D_i) \longrightarrow \mathcal{I}_{Z_f, \PP^2 \times \PP^2} \longrightarrow 0,
$$
where $D_i=d_ih_1+h_2$ for $1\le i \le m$, $B_j=(1+e_j-d) h_1+h_2$ for $1\le j\le m-2$ and $B_{m-1}=(d-1)h_1+2h_2$, and the multiplicative character of the Chern polynomial we get:
$$
c_t(\mathcal{I}_{Z_f, \PP^2 \times \PP^2})\cdot c_t(\oplus _{j=1}^{m-1}\sO_{\PP^2 \times \PP^2} (-B_j))=c_t(\oplus _{i=1}^m\sO_{\PP^2 \times \PP^2} (-D_i)).
$$

\medskip \noindent \medskip Looking at the coefficient of $t^2$ we obtain:
\[
\alpha \ h_1^2 +\beta h_1 \ h_2 + \gamma h_2^2+ ((d-1)\sum_{j=1}^{m-2}(1+e_j-d)+\sum_{1\leq i<j\leq m-2}(1+e_i-d)(1+e_j-d)) \ h_1^2+
\]
\[((m-1)\sum_{j=1}^{m-2}(1+e_j-d) +(m-2)(d-1)) \ h_1 \ h_2+\left({m-2 \choose 2}+2(m-2)\right) \ h_2^2=
\]
\[
\sum_{1\le i<j\le m}d_id_jh_1^2+(m-1)(d_1+\cdots +d_m)h_1h_2+{m\choose 2}h_2^2
\]
 or, equivalently,
 \begin{center} 
 $\displaystyle
\alpha +(d-1)\sum_{j=1}^{m-2}(1+e_j-d)+\sum_{1\leq i<j\leq m-2}(1+e_i-d)(1+e_j-d)=\sum_{1\leq i<j\leq m}d_id_j,$

$\displaystyle\beta +(m-1)\sum_{j=1}^{m-2}(1+e_j-d) +(m-2)(d-1) =(m-1)\sum_{i=1}^{m}d_i,$

$\displaystyle\gamma +{m-2 \choose 2}+2(m-2) ={m\choose 2}.$

\end{center}

Therefore, by \eqref{eq: m syzygy sequence}, \eqref{def e_j} and \eqref{def epsilon_j} we obtain:
$$\alpha =(d-1)^2-\tau(C), \ \beta = d-1 \ \text{ and } \ \gamma =1.$$

 \end{proof}

 \begin{theorem}\label{thm: Z_f irreducible}
 Let $C=V(f)$ be a reduced plane curve. Then 
\[
\mu(C)=\tau(C) 
\iff Z_f \text{\ is \ irreducible.}
\]
Moreover, if $Z_f$ is reducible, then $S_f \subsetneq Z_f$ and
\begin{equation}\label{eq: class of Z_f in terms of S_f}
Z_f = S_f + \sum_{i=1}^s m_i p_1^{-1} (P_i),
\end{equation}
where $\{P_1, \dots , P_s\} \subseteq {\rm Sing} \ C$,
the map $p_1: \PP^2 \times \PP^2\longrightarrow \PP^2$ denotes the first projection
and the integers $m_i \ge 1$ satisfy $\sum_{i=1}^s m_i= \mu (C) - \tau (C)$.

 \end{theorem}

 \begin{proof}
 The proof follows the lines of the proof of \cite[Theorem 3.3 and Theorem 5.6]{ABMR}.
 The first statement follows by observing that
$$
\mu(C)=\tau(C) 
 \iff \deg \nabla f = (d-1)^2 - \tau(C) \iff S_f=Z_f \iff Z_f \text{\ is \ irreducible.}
$$
 
 By comparing the classes \eqref{eq: class of ci} and \eqref{eq: class of Sf}, we see that the only possible irreducible components of $Z_f$ different from $S_f$ are cycles of class $m_i h_1^2$, for suitable coefficients $m_i \ge 1$. Since over $\PP^2 \setminus \Sigma_f$, the surfaces $S_f$ and $Z_f$ coincide by construction, we have $Z_f = S_f \cup \bigcup_{i=1}^s p_1^\star P_i$,
with $\{P_1, \cdots , P_s\} \subseteq {\rm Sing} \ C$.
 \end{proof}

 Putting all together we get the following efficient characterization of quasi-homogeneous isolated singularities of reduced plane curves in terms of the first syzygy matrix associated with the Jacobian ideal. 
 \begin{theorem} \label{main2}
 Let $C=V(f)$ be a reduced plane curve and $p\in Sing(C)$. Then, 
$$
p \text{\ is \ a \ quasi-homogeneous \ singularity } \iff \text {rk} \ M_f(p) \ge 1.
$$
 Therefore, it holds:
 $$
\mu(C)=\tau(C) 
\iff \text {rk} \ M_f(q) \ge 1, \forall \ q \in \PP^2.
$$
\end{theorem}
\begin{proof}
By Theorem \ref{thm: Z_f irreducible}, we have $\mu(C)=\tau(C) $
if and only if $Z_f$ is irreducible, and, by \eqref{eq: class of Z_f in terms of S_f},
 checking the irreducibility of $Z_f$ amounts to checking if for some point $p=(p_0:p_1:p_2) \in \Sigma_f$, the equations
 \eqref{eq: equations of Z_f} vanish identically, that is a first syzygy matrix satisfies $M_f(p)=0$.
 \end{proof}

\begin{remark}
    The geometric description given in Theorem \ref{thm: Z_f irreducible} implies, in particular, that the supports of the zero loci of the ideals $I_f$, spanned by the entries of $M_f$, and of $I_f + J_f$
    coincide.
\end{remark}
\begin{example} Let us reconsider the curves presented in Examples \ref{ex_msyz} and \ref{ex_QH} and apply our criteria.
\begin{enumerate}
\item The $3$-syzygy rational quintic $C\subset \PP^2$ of equation $C=V(f)=V(x_0^2 x_1^3+3 x_1^5-4 x_0 x_1^3 x_2+x_0^2 x_1 x_2^2)$ has two singularities in $p=(0:0:1)$ and $q=(1:0:0)$. Since 
$ \text {rk} \ M_f(p)=\text {rk} \ M_f(q)=1$, both singularities are quasi-homogeneous.

\item The $4$-syzygy rational sextic $C\subset \PP^2$ of equation $C=V(f)=V((x_0^2+x_1^2)^3-4 x_0^2 x_1^2 x_2^2)$ has three singularities in $p=(0:0:1)$, $q=(1:-i:0)$ and $r=(1:i:0)$. Moreover, we have
$ \text {rk} \ M_f(p)=0$ while $\text {rk} \ M_f(q)=\text {rk} \ M_f(r)=1$, thus $p$ is non-quasi-homogeneous, while $q$ and $r$ are quasi-homogeneous.

\item The $5$-syzygy rational sextic $C\subset \PP^2$ of equation $C=V(f)=V(x_1^6+x_0^2 x_1^2 x_2^2+x_2^6)$ has only one singular point in $p=(1:0:0)$. Since $\text {rk} \ M_f(p)=0$, the singularity is non-quasi-homogeneous.

\item As seen in Example \ref{ex_QH}, \eqref {ex_4syz_QH}, the nodal curve $C=V(f)=V(x_1^2x_2-x_0^2(x_0+x_2)$ has a unique singularity in $p=(0:0:1)$, which is quasi-homogeneous. Indeed, $\text {rk} \ M_f(p)= 1\neq 0$ as expected.

\item The $4$-syzygy rational quintic given by $C=V(f)=V(x_0 x_1^2 x_2^2 +x_1^5 + x_2^5)$, has a unique singularity in $p=(1:0:0)$; we have $\text {rk} \ M_f(p)= 0$ as expected, since the singularity is non-quasi-homogeneous.\\
\item The hypersurface in $V=V(f)\subset \mathbb{P}^n$ with $f=x_0^{a}x_1^{b}+x_2^d+\cdots +x_n^{d}$, with $d=a+b$ and $a, b\geq 2$ has as singular points $p=(1:0:\cdots:0)$ and $q=(0:1:0:\cdots:0)$. Since $V$ has a linear Jacobian syzygy of the type $\rho=(bx_0,-ax_1,0,\cdots,0)$ one has $\rho(p)\neq 0$ and $\rho(q)\neq 0$, and by Theorem \ref{syz} both singularities are quasi-homogeneous.

\item The hypersurface $V=V(f)\subset \mathbb{P}^3$ with $f=x_0^2x_3^3+x_1^4x_3+x_2^5-x_0x_1x_2x_3^2$ is a $9$-syzygy hypersurface with singular points $p=(0:0:0:1)$ and $q=(1:0:0:0)$, and $\rank M_f(p)=\rank M_f(q)=0$ so both singularities are non-quasi-homogeneous. \end{enumerate}

\end{example}

 \begin{example}
 The following examples of curves with non-quasi-homogeneous singularities are taken from \cite[Example 2.6]{NN}.\\
 
 (1) Consider the polynomial $f=x_0^4\,x_2-x_0^2\,x_1^2\,x_2+x_1^5
$. The curve has a singularity in $(0:0:1)$. A first syzygy matrix is given by
\begin{tiny}
 \[
 M_f= \begin{pmatrix}
5x_0x_1^2 - 2x_0x_1x_2 & 10x_0^2x_1 - 5x_1^3 - 2x_0^2x_2 & 5x_0^3 - 2x_0x_1x_2 & 0 \vphantom{5x_0x_1^2} \\
-2x_0^2x_2 & -2x_0x_1x_2 & -2x_0^2x_2 & x_0^4 - x_0^2x_1^2 \vphantom{-2x_0^2x_2} \\
-20x_1^2x_2 + 4x_1x_2^2 & -40x_0x_1x_2 + 8x_0x_2^2 & -20x_0^2x_2 - 10x_1^2x_2 + 4x_1x_2^2 & -5x_1^4 + 2x_0^2x_1x_2 \vphantom{-40x_0x_1x_2}
\end{pmatrix}
 \]
\end{tiny}
and it vanishes in $(0:0:1)$. Therefore, $(0:0:1)$ is not a quasi-homogeneous singularity.

\medskip
 (2) \label{ex: first example}
 Let $f=x_0^5\, x_2^3+x_0^3 \, x_1^5+x_1^7\, x_2$; the curve
 $C=V(f)$ has two singular points $p=(0:0:1)$ and $q=(1:0:0)$, and it is a $5$-syzygy curve. A first syzygy matrix is given by
 \begin{tiny}
 \[
 M_f= \begin{pmatrix}
15x_0^4 + 21x_0x_1^2x_2 & 5x_0x_1^5 - 21x_0^3x_2^3 & 25x_0^2x_1^4 + 147x_0x_1x_2^4 & 45x_0^3x_1^3 + 88x_1^5x_2 - 105x_0^2x_2^4 & -5x_1^6 + 21x_0^2x_1x_2^3 \\
-9x_0^3x_1 + 5x_1^3x_2 & -3x_1^6 - 5x_0^2x_1x_2^3 & -15x_0x_1^5 - 88x_0^3x_2^3 + 35x_1^2x_2^4 & -27x_0^2x_1^4 - 25x_0x_1x_2^4 & -9x_0^4x_2^2 + 5x_0x_1^2x_2^3 \\
-25x_0^3x_2 - 35x_1^2x_2^2 & 21x_1^5x_2 + 35x_0^2x_2^4 & 105x_0x_1^4x_2 - 245x_1x_2^5 & -75x_0^2x_1^3x_2 + 175x_0x_2^5 & 15x_0^2x_1^4 - 35x_0x_1x_2^4
\end{pmatrix}\]
 \end{tiny}
We easily check that the matrix $M_f$ vanishes only in $p=(0:0:1)$. So, the point $p$ is not a quasi-homogeneous singularity while the point $q$ is a quasi-homogeneous singularity.

 \end{example}
 
\begin{remark}
The equations of $S_f$ and its Cohen-Macaulayness when $S_f \neq Z_f$ are, in general, open questions.
 Our construction gives a way to answer such questions, in principle. Indeed, given a specific example, by a computer algebra symbolic system,
 like Macaulay 2, one can compute a primary decomposition of the ideal 
 of $Z_f$ in $\PP^2 \times \PP^2$; the unique prime factor 
 corresponding to a horizontal cycle gives the equations of $S_f$. Then the Cohen-Macaulayness can be checked by computing a minimal free resolution.
\end{remark}

 \begin{example}\label{ex: Ploski curves}
 A relevant family of curves with a non-quasi-homogeneous singular point is given by the {\it P\l oski curves}, that is unions of conics belonging to a hyperosculating pencil, or such unions plus the tangent line in the singularity. Such curves are free, see for instance \cite{D2} and \cite[Example 3.2(2)]{BMR}, they have only one singular point, and one has $\mu (C)=(d-1)^2-\lfloor \frac{d}{2}\rfloor$ and $\tau (C)=(d-1)(d-2)+1$.

 For instance,
 if we consider the degree $6$ P\l oski curve with equation 
 \[
 f=(x_0^2 + (x_0x_2 + x_1^2))(x_0^2 - (x_0x_2 + x_1^2))(x_0^2 + 2 (x_0x_2 + x_1^2)),
 \]
 we have $\mu(C)=22$ and $\tau(C)=21$. The primary decomposition of the ideal of $Z_f$ gives
 \begin{tiny}
 
 \[
 {}\left(x_1,\,x_0\right),
 \]
 \[(x_0 y_1-2 x_1 y_2,6 x_1^3 y_0 y_1-3 x_1^2 x_2 y_1^2-4 x_0^3 y_0 y_2+4 x_0 x_1^2 y_0 y_2+4 x_0^2 x_2 y_0 y_2+24 x_1^2 x_2 y_0 y_2+12 x_0 x_2^2 y_0 y_2-2 x_1^3 y_1 y_2-12 x_1 x_2^2 y_1 y_2+12 x_0^3 y_2^2
 \]
 \[
 \qquad \qquad +16 x_0 x_1^2 y_2^2+20 x_0^2 x_2 y_2^2-12 x_1^2 x_2 y_2^2-8 x_0 x_2^2 y_2^2-12 x_2^3 y_2^2,
 \]
 \[2 x_0^4 y_0-2 x_0^2 x_1^2 y_0-6 x_1^4 y_0-2 x_0^3 x_2 y_0-12 x_0 x_1^2 x_2 y_0-6 x_0^2 x_2^2 y_0+3 x_1^3 x_2 y_1-6 x_0^4 y_2-8 x_0^2 x_1^2 y_2
 \]
 \[
 \qquad \qquad +2 x_1^4 y_2-10 x_0^3 x_2 y_2+6 x_0 x_1^2 x_2 y_2+4 x_0^2 x_2^2 y_2+12 x_1^2 x_2^2 y_2+6 x_0 x_2^3 y_2),
 \]
 \end{tiny}
 
\medskip \noindent 
and we see that there is one vertical reduced component over the point $(0:0:1)$ and another irreducible one, which corresponds to $S_f$. Moreover, one can check that $S_f$ is Cohen-Macaulay in this case.

 On the other hand, if we consider the degree $8$ P\l oski curve with equation 
 \[
 f=(x_0^2 + (x_0x_2 + x_1^2))(x_0^2 - (x_0x_2 + x_1^2))(x_0^2 + 2 (x_0x_2 + x_1^2))(x_0^2 -2 (x_0x_2 + x_1^2)),\]
 we have $\mu(C)=45$ and $\tau(C)=43$. The primary decomposition of the ideal of $Z_f$ gives 
 \begin{tiny}
 \[(x_0 y_1-2 x_1 y_2,x_1^2,x_0 x_1,x_0^2),\]
 \[ (x_0 y_1-2 x_1 y_2,2 x_1^4 y_0 y_1^2-x_1^3 x_2 y_1^3+12 x_1^3 x_2 y_0 y_1 y_2-6 x_1^2 x_2^2 y_1^2 y_2-5 x_0^2 x_1^2 y_0 y_2^2-5 x_0^3 x_2 y_0 y_2^2+24 x_1^2 x_2^2 y_0 y_2^2+\]
 \[+8 x_0 x_2^3 y_0 y_2^2-12 x_1 x_2^3 y_1 y_2^2-4 x_0^4 y_2^3+10 x_1^4 y_2^3+25 x_0 x_1^2 x_2 y_2^3+15 x_0^2 x_2^2 y_2^3-8 x_2^4 y_2^3,\]
 \[4 x_1^5 y_0 y_1-2 x_1^4 x_2 y_1^2-5 x_0^3 x_1^2 y_0 y_2-5 x_0^4 x_2 y_0 y_2+24 x_1^4 x_2 y_0 y_2+24 x_0 x_1^2 x_2^2 y_0 y_2+8 x_0^2 x_2^3 y_0 y_2-12 x_1^3 x_2^2 y_1 y_2-4 x_0^5 y_2^2+\]
 \[+10 x_0 x_1^4 y_2^2+25 x_0^2 x_1^2 x_2 y_2^2+15 x_0^3 x_2^2 y_2^2-24 x_1^2 x_2^3 y_2^2-8 x_0 x_2^4 y_2^2,\]
 \[5 x_0^4 x_1^2 y_0-8 x_1^6 y_0+5 x_0^5 x_2 y_0-24 x_0 x_1^4 x_2 y_0-24 x_0^2 x_1^2 x_2^2 y_0-8 x_0^3 x_2^3 y_0+4 x_1^5 x_2 y_1+4 x_0^6 y_2\]
 \[-10 x_0^2 x_1^4 y_2-25 x_0^3 x_1^2 x_2 y_2-15 x_0^4 x_2^2 y_2+24 x_1^4 x_2^2 y_2+24 x_0 x_1^2 x_2^3 y_2+8 x_0^2 x_2^4 y_2).
\]
 \end{tiny}
 
In particular, we see that the vertical component over the point
$(0:0:1)$ is a non-reduced scheme, namely a double plane with double structure determined by the threefold of equation $x_0\,y_1-2\,x_1\,y_2$.

The equations of $S_f$ are given by the second ideal; it turns out that $S_f$ is not Cohen-Macaulay in this case.
 \end{example}


\section{Final comments}
 Another challenging question concerns the investigation of the geometric meaning of $\mu_p(C)$ and $\tau_p(C)$ in the non-quasi-homogeneous case, in terms of the Jacobian singular scheme $\Sigma_f$.

 A possible approach may involve the so-called 
 {\it Koszul hull} $\mathcal{K} \subset \PP^n \times \PP^n$ introduced in \cite[Definition 3.2.10 and Proposition 3.2.11]{B-C}; the scheme $\mathcal{K}$ is the standard determinantal scheme defined
 by the order two minors of the matrix
 \begin{equation}\label{eq: Koszul hull}
 \begin{pmatrix}
 y_0 & y_1 &\cdots &y_n\\
 \partial_0 f & \partial_1 f & \cdots & \partial_n f\\
 \end{pmatrix}.
 \end{equation}
 Such a scheme contains $Z_f$; being of the right codimension,
 it is Cohen-Macaulay, and a minimal free resolution of its ideal is given by the Eagon-Northcott
complex. 

Let us focus on the planar case $n=2$. It is easily seen that the numerical class
of $\mathbb {K}$ is $(d-1)^2 h_1^2 + (d-1) h_1 h_2 + h_2^2$.
As observed in \cite[Lemma 3.2.14]{B-C}, the surface $Z_f$ is linked to $p_1^\star \Sigma_f \equiv \tau(C) h_1^2$ in $\mathcal{K}$, in the sense that $\mathcal I_{Z_f} : \mathcal I_{\mathcal{K}}\cong p_1^\star \mathcal I_{\Sigma_f}$.

Similarly, one can consider the residual scheme to $S_f$ in $\mathcal{K}$; this is a codimension two cycle with numerical class
$\mu(C)h_1^2$, containing $p_1^\star \Sigma_f$, so it is of the form $p_1^\star {\mathcal V}$ for a suitable zero-dimensional scheme ${\mathcal V}\subset \PP^2$.

For instance, in the case of the example
\[f=(x_0^2 + (x_0x_2 + x_1^2))(x_0^2 - (x_0x_2 + x_1^2))(x_0^2 + 2 (x_0x_2 + x_1^2))
\]
treated in \ref{ex: Ploski curves}, the Jacobian ideal has degree $21$ and it is given by
\begin{tiny}
\[(6\,x_{0}^{5}+8\,x_{0}^{3}x_{1}^{2}-2\,x_{0}x_{1}^{4}+10\,x_{0}^{4}x_{2}-6\,x_{0}^{2}x_{1}^{2}x_{2}-6\,x_{1}^{4}x_{2}-4\,x_{0}^{3}x_{2}^{2}-12\,x_{0}x_{1}^{2}x_{2}^{2}-6\,x_{0}^{2}x_{2}^{3}
\]
\[
\,4\,x_{0}^{4}x_{1}-4\,x_{0}^{2}x_{1}^{3}-12\,x_{1}^{5}-4\,x_{0}^{3}x_{1}x_{2}-24\,x_{0}x_{1}^{3}x_{2}-12\,x_{0}^{2}x_{1}x_{2}^{2}
\]
\[
\,2\,x_{0}^{5}-2\,x_{0}^{3}x_{1}^{2}-6\,x_{0}x_{1}^{4}-2\,x_{0}^{4}x_{2}-12\,x_{0}^{2}x_{1}^{2}x_{2}-6\,x_{0}^{3}x_{2}^{2})
\]
\end{tiny}
while the saturation of the ideal of the degree $22$ scheme ${\mathcal V}$ is given by
\begin{tiny}
\[
(x_{0}^{4}x_{1}-x_{0}^{2}x_{1}^{3}-3\,x_{1}^{5}-x_{0}^{3}x_{1}x_{2}-6\,x_{0}x_{1}^{3}x_{2}-3\,x_{0}^{2}x_{1}x_{2}^{2}
\]
\[
\,3\,x_{0}^{5}+4\,x_{0}^{3}x_{1}^{2}-x_{0}x_{1}^{4}+5\,x_{0}^{4}x_{2}-3\,x_{0}^{2}x_{1}^{2}x_{2}-3\,x_{1}^{4}x_{2}-2\,x_{0}^{3}x_{2}^{2}-6\,x_{0}x_{1}^{2}x_{2}^{2}-3\,x_{0}^{2}x_{2}^{3}
\]
\[\,45\,x_{0}^{2}x_{1}^{4}+63\,x_{1}^{6}+34\,x_{0}^{3}x_{1}^{2}x_{2}+125\,x_{0}x_{1}^{4}x_{2}-19\,x_{0}^{4}x_{2}^{2}+69\,x_{0}^{2}x_{1}^{2}x_{2}^{2}+24\,x_{1}^{4}x_{2}^{2}+7\,x_{0}^{3}x_{2}^{3}+48\,x_{0}x_{1}^{2}x_{2}^{3}+24\,x_{0}^{2}x_{2}^{4})
\]
\end{tiny}
The properties
of such a scheme are mysterious, they could give a deeper insight into non-quasi-homogeneous singularities and deserve further investigation.

\end{document}